\documentclass[a4paper]{amsart}
\usepackage[utf8]{inputenc}
\usepackage[T1]{fontenc}
\usepackage{natbib}
\usepackage{amsmath,amssymb}
\usepackage{hyperref}
\usepackage{graphicx}
\usepackage{todonotes}

\usepackage{amssymb}
\usepackage{dsfont}
\usepackage{enumerate}
\usepackage[inline]{enumitem}
\usepackage{mathtools}
\usepackage{amsmath}
\usepackage{color}
\usepackage{bbm}         
\usepackage{amsfonts}
\usepackage{graphicx,hyperref,verbatim}
\usepackage{amsthm}
\usepackage{cleveref}
\usepackage{url}
\usepackage{tikz,multicol,manfnt}
\usepackage{qtree}
\usepackage[weather,clock,alpine]{ifsym}
\usepackage{xparse}
\usepackage{parnotes}
\usepackage{placeins}

\renewcommand{\phi}{\varphi}

\NewDocumentCommand{\set}{mg}{\left\{#1\IfNoValueF{#2}{\;\middle\vert\;#2}\right\}}

\newcommand{\ZZ}{\mathds Z}
\renewcommand{\geq}{\geqslant}

\DeclareMathOperator{\Aut}{Aut}    
\DeclareMathOperator{\supp}{supp}
\DeclareMathOperator{\ac}{ac}
\DeclareMathOperator{\lead}{lead}
\DeclareMathOperator{\rad}{rad}

\theoremstyle{plain}
\newtheorem*{thm}{Theorem}
\newtheorem{cor}{Corollary}
\newtheorem{lem}[cor]{Lemma}

\theoremstyle{definition}

\theoremstyle{remark}

\title{A radical answer to a question by Robinson}
\author{Blaise Boissonneau}

\address{Heinrich Heine University D\"usseldorf,
Faculty of Mathematics and Natural Sciences,
Mathematical Institute,
Universit\"atsstr.\ 1, 40225 D\"usseldorf, Germany.}

\email{blaise.boissonneau@hhu.de}

\author{Mikel E. Garciarena}

\email{mikel.eguzki.garciarena.perez@hhu.de}

\author{Immanuel Halupczok}
\email{math@karimmi.de}
\date{May 2026}
\thanks{The third author was partially supported by the research training group \emph{GRK 2240: Algebro-Geometric Methods in Algebra, Arithmetic and Topology}, funded by the DFG}

\begin{document}

\begin{abstract}
 We study the ring of Puiseux polynomials with integer coefficients. We prove notably that the order given by the leading coefficient is definable without parameters in the language of rings. This answers a question of R.~Robinson.
\end{abstract}

\maketitle

\section{Introduction}
R.~M.~Robinson states in \cite{undecrings}:

\begin{quotation} Mostowski and Tarski have shown that any ordered ring with a unity element 1, and with no element between 0 and 1, is undecidable. Here we have an additional primitive concept $<$, so that the result does not concern the undecidability of rings as such. However, it follows that any ring in which such an ordering is arithmetically definable (in terms of addition and multiplication) is undecidable. There does not seem to be any simple example of such a ring, except for the ring of integers, though the existence of other such rings can be proved.
\end{quotation}

\begin{figure}
\includegraphics[]{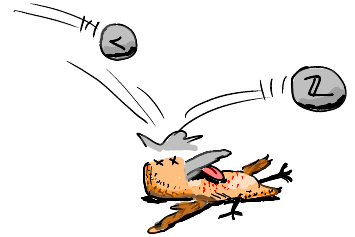}
\caption{
}\label{bird}
\end{figure}

To the best of our knowledge, no such ``simple example'' has previously been provided. We provide one here:

\begin{thm}
 Let $R=\bigcup_{n\in\mathds N_{>0}}\ZZ[t^\frac1n]$ be the ring of Puiseux polynomials with integer coefficients. The ordering given by $t>0$ and $t\gg1$ is discrete and $\emptyset$-$\mathcal L_{\rm{ring}}$-definable.
\end{thm}

As a corollary, as noted by Robinson, this ring is undecidable by \cite[Theorem 14]{tarski53}; however, as we will see, a step of the proof requires us to define $\ZZ$ in $R$, so this corollary is killing one bird with two stones, see Figure~\ref{bird}.

Regarding the comment by Robinson that ``the existence of other such rings can be proved'', it is unclear whether this simply refers to nonstandard models of $\ZZ$, or whether Robinson refers to some abstract proof of existence of rings with definable discrete ordering which are not elementarily equivalent to $\mathds Z$. We provide an abstract proof here: let $\phi(x,y)$ be a fixed formula
defining the ordering in $\mathds Z$ (say, by the sum of four squares), and consider the theory of rings such that $\phi$ defines a discrete ordering. This theory is finitely axiomatisable and has $\mathds Z$ as a model, hence, by Gödel's incompleteness theorem, it cannot be complete.

Regarding the ring $R$ of Puiseux polynomials, we found it of interest by itself, and hence decided to list some of its properties, whether or not they help us define the ordering. Note that $R$ is not elementarily equivalent to $\ZZ$, since there exist elements which cannot be written as plus or minus a sum of four squares, for example, $t-1$.

\subsection{Notations and conventions}
Unless stated otherwise, ``definable'' will mean ``definable in the language of rings without parameters''.

Let $x=\sum_{\alpha\in S}c_{\alpha}t^{\alpha}\in R$.
We define:
\begin{enumerate}
\item $\supp(x)=\set{\alpha\in\mathds Q_{\geqslant 0}}{c_\alpha\neq 0}$, the support of $x$;
\item $v(x)=\min(\supp(x))$, the valuation of $x$, with the additional convention $v(0)=+\infty$;
\item $\ac(x)=c_{v(x)}$, the angular component of $x$, with the additional convention $\ac(0)=0$;
\item $\deg(x)=\max(\supp(x))$, the degree of $x$, and $\lead(x)=c_{\deg(x)}$, the leading coefficient of $x$, with the additional conventions $\deg(0)=-\infty$ and $\lead(0)=0$, respectively.
\end{enumerate}

For any $x,y\in R$, we have $v(xy)=v(x)+v(y)$ and $\ac(xy)=\ac(x)\ac(y)$.

By construction, any $x\in R\setminus\set0$ can be factorized uniquely as $t^{v(x)}(\ac(x)+y)$, with $v(y)>0$.
\section{Defining powers of \texorpdfstring{$t$}{t}}

\begin{lem}
 For $x\in R\setminus\set0$, the following are equivalent:
 \begin{enumerate}
 \item $x$ is a power of $t$, that is, $\exists\alpha\in\mathds Q_{\geqslant 0}$ such that $x=t^\alpha$;
 \item $x$ admits a square root and an $n$th root for infinitely many $n\in\mathds N$;
 \item $x$ admits an $n$th root for all $n\in\mathds N$.
 \end{enumerate}
\end{lem}

\begin{proof}
$(1)\Rightarrow(3)\Rightarrow(2)$ is immediate. Now take $x\in R\setminus\set0$ admitting $n$th-roots for infinitely many $n\in\mathds N$. For any such $n$, let $y_n\in R$ be such that $y_n^n=x$. Now $\ac(y_n^n)=\ac(y_n)^n=\ac(x)$. So $\ac(x)$ has a square root and an $n$th root for infinitely many $n$, which means $\ac(x)=1$ and $\ac(y_n)=\pm1$.

Factorize $x$ into $t^{v(x)}(1+y)$, with $v(y)>0$. Let $n>|\ac(y)|$ such that $x$ admits an $n$th root. Let $y_n^n=x$, and write $y_n=t^{v(y_n)}(\varepsilon+z)$, where $\varepsilon=\ac(y_n)=\pm1$ and $v(z)>0$. Since $\varepsilon^n=\ac(x)=1$, we have $(\varepsilon+z)^n=1+n\varepsilon^{n-1}z+w$, where $z^2|w$. Hence:
$$x=y_n^n=t^{nv(y_n)}(\varepsilon+z)^n=t^{v(x)}(1+n\varepsilon^{n-1}z+w)=t^{v(x)}(1+y)$$

that is, $y=n\varepsilon^{n-1}z+w$. If $z\neq 0$, then $\ac(y)=n\varepsilon^{n-1}\ac(z)$. Thus $|\ac(y)|=n|\ac(z)|\geq n$, contradicting the choice of $n$. Therefore $z=0$, and consequently $y=0$. Hence $x=t^{v(x)}$.

\end{proof}

\begin{cor}
 The set of powers of $t$ is definable in $R$ by the following formula:
 $$\phi(x)\colon(\exists z\;x=z^2\wedge\forall y\; (y|x\rightarrow (\exists z\;y=\pm z^2)))$$
 Furthermore, the map $R\setminus\set{0}\rightarrow R$, $x\mapsto t^{v(x)}$ is definable.
 
 \noindent Finally, $\Aut(R)\cong\mathds Q_{>0}^\times$. More precisely, any automorphism of $R$ is of the form $t\mapsto t^\alpha$, for some $\alpha>0$.
\end{cor}

\begin{proof}
 If $x\in R$ satisfies $\phi$, then $x$ is a square, and any divisor of $x$ is itself, up to sign, a square. In particular, $x$ has a $2^n$th root for every $n$, thus by the previous lemma $x$ must be a power of $t$.

 Conversely, if $x=t^\alpha$, then $x$ is a square since $t^{\frac\alpha2}\in R$. We claim that every divisor of $x$ is of the form $\pm t^\beta$ for $\beta\leqslant\alpha$. 
 Indeed, if $x=yz$ for some $y,z\in R$, then there is $n\in\mathds N$ such that $x,y,z\in\ZZ[t^{\frac1n}]$, in which the only divisors of $t^{\alpha}$ are of the form $\pm t^{\beta}$. Therefore, all the divisors of $t^\alpha$ are themselves plus or minus a square, and $x$ satisfies $\phi$.

 To define the (graph of the) map $x\mapsto t^{v(x)}$, we can use the formula $$\psi(x,y)\colon x\neq0\wedge\phi(y)\wedge y|x\wedge(\forall z\;((\phi(z)\wedge z|x)\rightarrow z|y))$$ 

 Finally, it is clear that for each $\alpha\in\mathds Q_{> 0}$ there is a unique automorphism of $R$ sending $t^\beta$ to $t^{\alpha\beta}$ for every $\beta \in  \mathds Q_{\geqslant 0}$. Conversely, if $\sigma \in \operatorname{Aut}(R)$, then
 for every $\beta \in \mathds{Q}_{>0}$, 
  $\sigma(t^\beta)$ has an
$n$th root for every $n$, hence by Lemma~1 we have
$\sigma(t^\beta)=t^{\rho(\beta)}$ for some $\rho(\beta) \in \mathds{Q}_{\geq 0}$.
That $\sigma$ respects multiplication on the set of powers of $t$ implies that $\rho(\beta) = \alpha\cdot \beta$ for some $\alpha \in \mathds{Q}_{\geq 0}$.
Moreover, $\alpha \neq 0$, since automorphisms preserve units, and $t$
is not a unit whereas $t^0=1$ is. Thus $\alpha \in \mathds{Q}_{>0}$.
Thus $\sigma$ is necessarily the automorphism sending $t$ to $t^\alpha$.
\end{proof}

\section{Defining \texorpdfstring{$\ZZ$}{ℤ}}

\begin{lem} In $R$, the set of polynomials of degree $\leqslant 0$, that is, $\ZZ$, is definable.
\end{lem}

\begin{proof}
We first define the set of signed powers of 2 with the formula
$$\phi(x)\colon x\neq0\wedge(\forall y\;((y|x\wedge y\neq\pm1)\rightarrow 2|y))$$
We then define the positive powers of 2 by
$$\phi_+(x)\colon\phi(x)\wedge\exists z\; (x=z^2\vee x=2z^2)$$
Finally, recall that any odd integer $x$ satisfies $\exists k\exists n\;(k\neq0\wedge n>0\wedge xk+1=2^n)$, by Euler-Fermat's theorem \cite[Theorema 11]{eulertotient}. Therefore, in $R$, the formula
$$\psi(x)\colon\exists k\;(k\neq0\wedge\phi_+(xk+1))$$
is satisfied by the odd integers, by 0, and by no element of degree $>0$: if $\deg(x)>0$, then for any $k\neq0$ we have $\deg(xk)=\deg(x)+\deg(k)>\deg(k)\geqslant0$, and $\deg(xk+1)=\deg(xk)>0$, so $xk+1$ cannot be a power of 2.

Thus, every element satisfying $\psi$ belongs to $\ZZ$. Since every integer is a sum of two elements which are either odd integers or 0, $\ZZ$ can be defined in $R$ by
$\exists y\exists z\;(\psi(y)\wedge\psi(z)\wedge x=y+z)$.
\end{proof}

\section{Defining the degree}

\begin{lem}
 Let $x\in R$, $n\in\mathds Z$ and $m\in\mathds Q_{\geqslant 0}$. Assume $|n|>|\lead(x)|$. Then $m\geqslant\deg(x)$ iff $\exists\alpha\in \mathds Q_{>0}$, $\exists g\in R$ and $\exists c\in\mathds Z$ such that $x=(nt^\alpha-1)g+ct^m$.
\end{lem}

\begin{proof}
 Suppose $x=(nt^\alpha-1)g+ct^m$. If $g=0$, then $x=ct^m$ and $m=\deg(x)$. Assume now $g\neq0$. We have $\lead((nt^\alpha-1)g)=n\lead(g)$. If $m<\deg(g)+\alpha$, then $\lead(x)=\lead((nt^\alpha-1)g+ct^m)=n\lead(g)$ which contradicts the choice of $n$. Thus $m\geqslant\deg(g)+\alpha$ and $\deg(x)=\deg((nt^\alpha-1)g+ct^m)\leqslant\max(\deg(g)+\alpha,m)=m$.

 Suppose now that $m\geqslant\deg(x)$.
 There is $r\in\mathds N$ such that $x\in\mathds Z[t^{\frac1r}]$ and such that $rm \in \mathds N$.
 By applying the automorphism $\sigma\colon t\mapsto t^{r}$, we may assume that
$x \in \mathds Z[t]$ and $m \in \mathds N$.
 
 Let $d=\deg(x)$ and write $x=\sum_{0\leqslant i\leqslant d}a_it^i$. We can now find $g\in\mathds Z[t]$ such that $x=(nt-1)g(t)+ct^{m}$. Indeed, the system
 $$\begin{pmatrix}
 -1&0&\dots&\dots&\dots&0\\
 n&-1&\ddots&&&\vdots\\
 0&\ddots&\ddots&\ddots&&\vdots\\
 \vdots&\ddots&\ddots&\ddots&\ddots&\vdots\\
 \vdots&&\ddots&n&-1&0\\
 0&\dots&\dots&0&n&1
 \end{pmatrix}
 \times
 \begin{pmatrix}
 b_0\\\vdots\\\vdots\\\vdots\\b_{m-1}\\c
 \end{pmatrix}
 =
 \begin{pmatrix}
 a_0\\\vdots\\a_d\\0\\\vdots\\0
 \end{pmatrix}
 $$
 admits a solution (in $\mathds Z$), and we can take $g(t)=\sum_{0\leqslant i<m}b_it^i$.

 Now $x=(nt-1)g(t)+ct^{m}$, thus $x=(nt^{\alpha}-1)g+ct^m$ with $\alpha=1$.
\end{proof}

Recall that in $\mathds Z$, the ordering is definable, for example by using Lagrange's four-square theorem 
\cite{4squares}.

\begin{cor}
 The degree map $R\setminus\set0\rightarrow R$, $x\mapsto t^{\deg(x)}$ and the leading coefficient map $R\rightarrow\mathds Z$, $x\mapsto\lead(x)$ are definable.
\end{cor}

\begin{proof}
 For the degree map, consider first $\phi(x,y,z,n,c)\colon \exists g\;x=(nz-1)g+cy$. Now the formula $$\psi(x,y)\colon y\in t^{\mathds Q_{\geqslant 0}}\wedge(\exists N\in\mathds Z\;\forall n>N\;\exists z\in t^{\mathds Q_{\geqslant 0}}\,\exists c\in\mathds Z\;\phi(x,y,z,n,c))$$
 where ``$>$'' is the ordering on $\mathds Z$, defines the set of $y=t^m$ with $m\geqslant
 \deg(x)$. We can now take the minimal such -- with respect to the divisibility relation.

 For the leading coefficient map, we have $a=\lead(x)$ iff $x=0\wedge a=0$ or $a\in\mathds Z\wedge\deg(x-at^{\deg(x)})\neq\deg(x)$, which is a definable condition since the degree map is definable.
\end{proof}

\section{Defining the ordering}

We are now ready to prove the theorem. Consider the ordering on $R$ with $t>0$ and $t\gg1$, that is, $x>0$ iff $\lead(x)>0$. This ordering is clearly definable, simply because $x\mapsto\lead(x)$ is a definable map, and because the ordering on $\mathds Z$ is definable.

\section{Defining almost anything}

For any $x\in R$, there is $\sigma\in\Aut(R)$ such that $\sigma(x)\in\mathds Z[t]$. For any such $\sigma$, let $\alpha\in\mathds Q_{>0}$ be such that $\sigma(t)=t^\alpha$. We call such an $\alpha$ a \emph{turnix} of $x$,\footnote{So called because one can turn $x$ into a proper polynomial by evaluating it at $t^\alpha$.} and the minimal turnix of $x$ is called the \emph{radix} of $x$.
The radix of $x$ always exists if $x\notin\mathds Z$; if $x\in\mathds Z$, we say that it is of radix $0$. Note that if $\rad(x)\neq 0$, then turnices of $x$ are exactly the integer multiples of $\rad(x)$.

\begin{lem}
 Let $x\in R$ and let $\alpha$ be a turnix of $x$. Let $\sigma$ be the automorphism sending $t$ to $t^\alpha$, and write $\sigma(x)=f(t)\in\ZZ[t]$. The ``evaluation'' map $E_{x,\alpha}\colon\ZZ\rightarrow \ZZ$, $n\mapsto f(n)$ is definable uniformly with parameters $x$ and $t^{\frac1\alpha}$ in $R$.
\end{lem}

\begin{proof}
 We have $E_{x,\alpha}(n)=c$ iff $f(n)=c$ iff $f(t)-c$ has a root at $n$ iff $t-n$ divides $f(t)-c$ in $\ZZ[t]$. Since $\sigma(x)=f(t)$, we have $x=\sigma^{-1}(f(t))=f(t^{\frac1\alpha})$. Thus, $t-n$ divides $f(t)-c$ in $\ZZ[t]$ iff $t^{\frac1\alpha}-n$ divides $x-c$ in $R$, so (the graph of) $E_{x,\alpha}$ is defined by
 $$\phi(n,c;x,t^{\frac1\alpha})\colon (n,c\in\ZZ)\wedge((t^{\frac1\alpha}-n) | (x-c))$$
\end{proof}

\begin{lem}
The set of turnices $\set{(x,t^{\frac1\alpha})\in (R\setminus\ZZ)^2}{\alpha\text{ turnix of }x}$ and the radix map $R\setminus\mathds Z\rightarrow t^{\mathds Q_{>0}}$, $x\mapsto t^{\frac1{\rad(x)}}$ are definable.
\end{lem}

\begin{proof}
 We claim that $\alpha$ is a turnix of $x$ iff for all $n\in\ZZ$, there exists a unique $c\in\ZZ$ such that $t^{\frac1\alpha}-n$ divides $x-c$. Indeed, if $\alpha$ is a turnix of $x$, then by the previous lemma $E_{x,\alpha}$ is definable, and each $n\in\ZZ$ can be sent to a unique $c\in\ZZ$ witnessing the property above, namely, $c=E_{x,\alpha}(n)$. 
 
 Conversely, assume $\alpha$ is not a turnix of $x$. For a given $n$, there could be $c$ such that $t^{\frac1\alpha}-n$ divides $x-c$, for example, $x=t^{\frac32}$, $\alpha=1$, $n=0$, which yields $c=0$. Our claim is that this will only work for finitely many $n$.
 
 Assume that $(n,c)$ is a pair of integers such that $(t^{\frac1\alpha}-n)|(x-c)$, thus, there exists $y\in R$ with $x=(t^{\frac1\alpha}-n)y+c$. Write $x=\sum_{i=0}^d a_i t^{q_i}$, with $d\in\mathds N$, $a_i\in\mathds Z$ and $q_i\in
 \mathds Q_{\geqslant 0}$. We split $x$ in two, according to whether the exponents $q_i$ are multiples of $\frac1\alpha$ or not:
 $$x_{\alpha}=\sum_{\mathclap{q_i\in\frac 1\alpha\mathds N}}a_i t^{q_i},\;\;\;\overline x_{\alpha}=x-x_\alpha.$$ Likewise, we split $y$ into $y_\alpha+\overline y_{\alpha}$. If $\alpha$ is not a turnix of $x$, then $\sigma(x)=\sum_{i=0}^d a_it^{q_i\alpha}$ has a non-integer exponent; that is, some $q_i$ is not a multiple of $\frac1\alpha$ and $\overline x_{\alpha}$ is non-zero.

 Now from $x=(t^{\frac1\alpha}-n)y+c$ we deduce $x_{\alpha}=(t^{\frac1\alpha}-n)y_{\alpha}+c$ and $\overline x_{\alpha}=(t^{\frac1\alpha}-n)\overline y_{\alpha}$. If $n\neq0$, then we have $\ac(\overline x_{\alpha})=-n\ac(\overline y_{\alpha})$, thus, either $n=0$, or $n$ divides $\ac(\overline x_{\alpha})$. For given $x$ and $\alpha$, there can only be finitely many such $n$.

\ 

Therefore the set of turnices can be defined by $$\phi(x,y)\colon x\notin\ZZ\wedge y\in t^{\mathds Q_{>0}}\wedge\forall n\in\ZZ\;\exists! c\in\ZZ\;((y-n)|(x-c))$$
and the radix map by
\[\psi(x,y)\colon\phi(x,y)\wedge\forall z(\phi(x,z)\rightarrow z|y).\qedhere\]
\end{proof}

This very powerful machinery allows us to obtain the coefficients and the exponents of any element in $R$. To do so, first recall that in $\mathds Z$ and therefore in $R$ too, the structure $(\mathds Z^{<\omega},+,\cdot)$ as well as the map $(\mathds N,\mathds Z^{<\omega})\rightarrow\mathds Z$, $(n,(a_i)_i)\mapsto a_n$ (where $a_n=0$ if $n$ is too big), are interpretable via Gödel's $\beta$ function \cite{godel-beta}. We will consider $\mathds Z^{<\omega}$ as an imaginary sort equipped with the structure mentioned above.

Given an element $x\in R\setminus\mathds Z$, there is a unique tuple $(a_i)_{i\leqslant d}\in\mathds Z^{<\omega}$ such that $x=\sum_{i=0}^d a_it^{\frac i{\rad(x)}}$. Combining this with the definability of the radix map, we obtain:
\begin{cor}
 The map $R\setminus\mathds Z\rightarrow\ZZ^{<\omega}$, $\sum_{i=0}^d a_it^{\frac i{\rad(x)}}\mapsto (a_i)_{i\leqslant d}$ is definable.

 The map $\mathds N\times(R\setminus\mathds Z)\rightarrow\ZZ\times t^{\mathds Q_{\geqslant 0}}$, $\sum_{i=0}^d a_it^{\frac i{\rad(x)}}\mapsto a_nt^{\frac n{\rad(x)}}$ is definable.
\end{cor}

Thus, any notion which can be expressed in terms of the coefficients and the exponents of a polynomial is now definable.

\section{Other rings}
The reason why we looked at the ring $\bigcup_n\mathds Z[t^{\frac1n}]$ is that, if you want a discretely ordered ring which is not $\mathds Z$, it must have an infinite positive element $t$, giving you the ring $\mathds Z[t]$. However, the ordering is not definable in this ring, as $t\mapsto -t$ is an automorphism. Note that our first stone still works, that is, $\mathds Z$ is definable through powers of 2, and this ring is undecidable. The secret sauce that makes $R$ a really powerful object is our ability to define powers of $t$ in $R$.

From $\mathds Z[t]$, in order to distinguish $t$ and $-t$, you can simply add a square root of $t$. But now $t^{\frac12}$ and $-t^{\frac12}$ are indistinguishable. So, do it again, and again, and again, until you obtain... $R'=\bigcup_n\mathds Z[t^{\frac1{2^n}}]$. This is not the ring we presented in this paper, as we thought that the polynomial ring with any rational exponent is a more natural object than restricting to only dyadic exponents, but our proof would work alike for the ring $R'$; in fact, it would work alike for any ring lying between $R'$ and $R$, giving infinitely many examples in answer to Robinson.

\bibliographystyle{alpha}
\bibliography{ref.bib}

\end{document}